\documentclass[a4paper,12pt,fleqn]{article}
\usepackage[latin1]{inputenc}  \usepackage[T1]{fontenc}
\usepackage{lscape}
\usepackage{tabularx}     
\usepackage{pstricks}    
\usepackage{rotating}
\usepackage{amsfonts,amsmath,amsthm,amssymb,dsfont,bbold}
\usepackage{marginnote}
\usepackage{rotate}
\usepackage{pgfplots}
\usepackage{tikz}
\usepackage{tikz-cd}
\usepackage{todonotes}
\usepackage[arrow,matrix,curve]{xy} 	
\title{Improving a Constant in High-Dimensional Discrepancy Estimates}
\author{Hendrik Pasing \and Christian Weiß}
\date{\today}

\newtheorem{thm}{Theorem}[section]

\newtheorem{lem}[thm]{Lemma}
\newtheorem{prop}[thm]{Proposition}

\makeindex

\begin{document} 

\maketitle

\begin{abstract}% englische Fassung
	For all $s \geq 1$ and $N \geq 1$ there exist sequences $(z_1,\ldots,z_N)$ in $[0,1]^s$ such that the star-discrepancy of these points can be bounded by
	$$D_N^*(z_1,\ldots,z_N) \leq c \frac{\sqrt{s}}{\sqrt{N}}.$$
	The best known value for the constant is $c=10$ as has been calculated by Aistleitner in \cite{Ais11}. In this paper we improve the bound to $c=9$. 
\end{abstract}

\section{Introduction}
When Quasi-Monte Carlo methods are applied in practice to answer financial mathematical questions, the occuring problems frequently involve to explicilty or implicitly calculate integrals. Often the arithmetic mean of some function evaluations $f(z_1),\ldots,f(z_N)$ is taken as an approximation of the integral under consideration. A theoretic justification for this approach is the Koksma-Hlawka inequality which states that the difference between the arithmetic mean and the integral of a function $f$ over the $s$-dimensional unit cube is bounded by the product of the total variation of $f$ in the sense of Hardy-Krause and the so-called star-discrepancy $D_N^*(z_1,\ldots,z_N)$.\\[12pt] 
Because the problems occuring in practice are in addition typically high-dimensional ($s \gg 0$) and function evaluation is expensive ($N \ll \infty$), see e.g. \cite{BFW14}, \cite{KNK18}, classical low-discrepancy sequences which satisfy the inequality
$$D^*_N(z_1,\ldots,z_N) \leq c \frac{(\log N)^{s-1}}{N}$$
are of limited use. This observation is known as the \textit{coarse of dimensionality}, compare e.g. \cite{Nie92}, Chapter~1, \cite{Slo09}. Instead, it is hence desireable to construct sequences which have a small star-discrepancy if $N$ is small in comparison to $s$.\\[12pt]
In \cite{HNWW01} it was shown that for every $s\geq 1$ and $N \geq 1$, there exists a finite sequence $(z_1,...,z_N)$ of elements of the $s$-dimensional unit cube such that the star-discrepancy of this sequence satisfies
 $$D_N^*(z_1,...,z_N) \leq c \frac{\sqrt{s}}{\sqrt{N}}$$
for some constant $c$ independent of $s$ and $N$. However, no concrete value for $c$ was calculated in this paper. In \cite{Ais11}, a new proof of the result was given including the explicit upper bound $c=10$. In this paper, we improve the upper constant to $c=9$. More precisely, we show:
\begin{thm} \label{thm:main_thm} For any $s \geq 1$ and $N \geq 1$, there exists a sequence $(z_1,\ldots,z_n)$ of elements of the $s$-dimensional unit cube such that
	\begin{align} \label{inequ:main}
D_N^*(z_1,\ldots,z_N) < 9\frac{\sqrt{s}}{\sqrt{N}}. 
\end{align}
\end{thm}
An improvement of the constant $c$ is of important practical use: for bounding the discrepancy of a sequence $(z_1,\ldots,z_N)$ deterministically by $1$ we need a sequence $(z_1,\ldots,z_N)$ of length $N > c^2 s$. So $N$ depends on $c^2$ which means it is a matter of interest to find the best possible theoretical value of $c$. Our work is a contribution to this aim.\\[12pt] 
Our proof closely follows the one presented in \cite{Ais11}. As was already mentioned therein, an improvement of Gnewuch's upper bound for the smallest cardinality of a $\delta$-cover, Theorem~\ref{thm:gnewuchalt}, should result in a better value for $c$. Indeed, in Proposition~\ref{prop:improvement_gnewuch} we are able to improve Gnewuch's result. Afterwards, we only need to slightly amend Aistleitner's proof for the new upper bound and find a lower value for $c$.\\[12pt] %still can apply Bernstein's inequality to improve the value for $c$.\\[12pt] 
Finally, it should be mentioned, that the rate of convergence $\sqrt{s/N}$ is in some sense best possible: it was shown in \cite{Doe11} by Doerr that for a random set of independent, uniformly distributed points $(z_1,\ldots,z_N)$ the inequality
$$E(D_N^*(z_1,...,z_N)) \geq \widetilde{c} \frac{\sqrt{s}}{\sqrt{N}}$$
for the expected value holds. %Although not explicitly mentioned in \cite{Doe11}, the estimation $\widetilde{c} \geq \tfrac{3}{5120}$ can be implicitly deduced from the proof. \todo{@Hendrik: Bitte pr\"ufen!!!} 

%
%Habe ich geprüft, siehe PDF, HP
%

 \section{Proof of the main results}
 
Before we come to the proof of Theorem~\ref{thm:main_thm} we collect some of the necessary background information. 

\paragraph{Discrepancy.} Let $Z=(z_n)_{n \geq 0}$ be a sequence in $[0,1)^s$. Then the \textbf{star-discrepancy} of the first $N$ points of the sequence is defined by
$$D^*_N(Z) := \sup_{B \subset [0,1)^d} \left| \frac{A_N(B)}{N} - \lambda_s(B) \right|,$$
where the supremum is taken over all intervals $B = [0,a_1) \times [0,a_2) \times \ldots \times [0,a_s) \subset [0,1)^s$ and $A_N(B) :=  |\left\{ n \ \mid \ 0 \leq n < N, z_n \in B \right\}|$ and $\lambda_s$ denotes the $s$-dimensional Lebesgue-measure. If $D^*_N(Z)$ satisfies
$$D_N(Z) = O(N^{-1}(\log N)^{s-1})$$
then $Z$ is called a \textbf{low-discrepancy sequence}. For more details we refer the reader to \cite{Nie92}.
\paragraph{$\delta$-bracketing Covers.} In this paper we will use the notation of $\delta$-bracketing covers which was introduced in \cite{Gne08}: let  $\mathcal{F} \subset L^1([0,1]^s)$ be a subset of the real valued Lebesgue integrable functions. For $0 < \delta \leq 1$ and $f, g \in \mathcal{F}$ with
$$\int_{[0,1]^s} (g(x) - f(x)) \mathrm{d} x \leq \delta.$$ 
we call the set
$$[f,g]_{\mathcal{F}} := \left\{ h \in \mathcal{F} \ | \ f \leq h \leq g \right\}$$
a \textbf{$\delta$-bracket of $\mathcal{F}$}. A finite subset $\Gamma \subset \mathcal{F}$ is called a \textbf{$\delta$-cover of $\mathcal{F}$}, if for every $h \in \mathcal{F}$, there exists $f,g \in \Gamma$ with $h \in [f,g]_\mathcal{F}$. A \textbf{$\delta$-bracketing cover of $\mathcal{F}$} is a set of $\delta$-brackets whose union is $\mathcal{F}$. The number $\mathcal{N}(\mathcal{F},d)$ denotes the smallest cardinality of a $\delta$-cover of $\mathcal{F}$, i.e.
$$\mathcal{N}(\mathcal{F},d) := \min \left\{ |\Gamma| \ | \ \textrm{$\Gamma$ is a $\delta$-cover} \right\}.$$
Similarly, $N_{[\ ]}(\mathcal{F},\delta)$ denotes the smallest cardinality of a $\delta$-bracketing cover. In the following we will restrict to the specific subset of $\mathcal{F}$ which consists of all indicator function of the form $\mathbb{1}_{[0,x)}$ for some $x < 1$ and use the notation $\mathcal{N}(s,d)$ and $N_{[\ ]}(s,\delta)$ in this case. 
\paragraph{Gnewuch's inequality.} In \cite{Gne08}, Gnewuch proved the following inequality for $N_{[\ ]}(s,\delta)$.
\begin{thm}[\cite{Gne08}, Theorem 1.15] \label{thm:gnewuchalt} Let $s \in \mathbb{N}$ and $0 < \delta \leq 1$. Then
	\begin{align} \label{eq:gnewuch}
	N_{[\ ]}(s,\delta) \leq 2^{s-1} \left( \frac{s^s}{s!}\right) \left( \delta^{-1}+1\right)^s.
	\end{align}
\end{thm}
We will focus here on an intermediate result of Gnewuch which he derived during the proof of Theorem~\ref{thm:gnewuchalt} and state it as a lemma. Afterwards we will show that it can be used to strengthen inequality~\eqref{eq:gnewuch}.
\begin{lem}\label{lemma_gnewuch} Let $s \in \mathbb{N}$ and $0 < \delta \leq 1$. Then
$$N_{[\ ]}(s,\delta) \leq \sum\limits_{k=0}^{s-2} {s \choose k+1} 2^{s-k-2} \frac{s^s}{(s-k)!} \left( \delta^{-1}+\frac{1}{2} \right)^{s-k}+ \delta^{-1}+1.$$
\end{lem}
Indeed, we prove here the following stronger version of Gnewuch's inequality. Note that
$$2^{s-2} \left( \frac{s^s}{s!}\right) \left( \delta^{-1}+1\right)^s+\frac{1}{2}\left( \delta^{-1}+1\right) < 2^{s-1} \left( \frac{s^s}{s!}\right) \left( \delta^{-1}+1\right)^s$$
for all $s \geq 2$.
\begin{prop}[Upper bound for covering numbers] \label{prop:improvement_gnewuch} Let $s \in \mathbb{N}$ and $0 < \delta \leq 1$. Then
	$$N_{[\ ]}(s,\delta) \leq 2^{s-2} \left( \frac{s^s}{s!}\right) \left( \delta^{-1}+1\right)^s+\frac{1}{2}\left( \delta^{-1}+1\right).$$
\end{prop}
\begin{proof} We prove our claim by induction on $s$. Let $n:=\lceil \delta^{-1} \rceil$. For $s=1$ we have 
	$$N_{[\ ]}(s,\delta) \leq n \leq \delta^{-1}+1.$$ 
  So let$s \geq 2$. With Lemma~\ref{lemma_gnewuch} we have
	\begin{align*}
	N_{[\ ]}(s,\delta) &\leq \sum\limits_{k=0}^{s-2} {s \choose k+1} 2^{s-k-2} \frac{s^s}{(s-k)!} \left( \delta^{-1}+\frac{1}{2} \right)^{s-k}+ \delta^{-1}+1 \\
	&\leq \sum\limits_{k=0}^{s-2} {s \choose k} 2^{s-k-2} \frac{s^s}{s!} \left( \delta^{-1}+\frac{1}{2} \right)^{s-k}+ \delta^{-1}+1
	\end{align*}
	For the right hand side we get
	\begin{eqnarray*}
		&& \sum\limits_{k=0}^{s-2} {s \choose k} 2^{s-k-2} \frac{s^s}{s!} \left( \delta^{-1}+\frac{1}{2} \right)^{s-k}+ \delta^{-1}+1 \\
		&=& \sum\limits_{k=0}^{s} {s \choose k} 2^{s-k-2} \frac{s^s}{s!} \left( \delta^{-1}+\frac{1}{2} \right)^{s-k} -\sum\limits_{k=s-1}^{s} {s \choose k} 2^{s-k-2} \frac{s^s}{s!} \left( \delta^{-1}+\frac{1}{2} \right)^{s-k} \\
		&& \quad+ \delta^{-1}+1 
	\end{eqnarray*}
	Finally
	$$-\sum\limits_{k=s-1}^{s} {s \choose k} 2^{s-k-2} \frac{s^s}{s!} \left( \delta^{-1}+\frac{1}{2} \right)^{s-k}+ \delta^{-1}+1 \leq \frac{1}{2}\left( \delta^{-1}+1 \right)$$
	and
	$$\sum\limits_{k=0}^{s} {s \choose k} 2^{s-k-2} \frac{s^s}{s!} \left( \delta^{-1}+\frac{1}{2} \right)^{s-k} = 2^{s-2}\frac{s^s}{s!}\left( \delta^{-1}+1 \right)^{s}$$
	imply
	$$N_{[\ ]}(s,\delta) \leq 2^{s-2} \left( \frac{s^s}{s!}\right) \left( \delta^{-1}+1\right)^s+\frac{1}{2}\left( \delta^{-1}+1\right).$$
\end{proof}

\begin{proof}[Proof of Theorem~\ref{thm:main_thm}] We closely follow the proof of \cite{Ais11}, Theorem~1, and amend the arguments therein in order to take into account our improved version of Gwenuch's inequality. For $s=1$, the points of distance $1/n$ satisfy the inequality and for $s=2$, the Hammersley sequence with base $2$ does the job. Therefore let $s \geq 3$. We may without loss of generality assume that $N > 81s$ because the claim trivially follows otherwise. For a clear presentation, we subdivide our proof into 5 steps. Since steps 1,2 and 5 are essentially the same as in \cite{Ais11}, we will not go into details here but still present them for the sake of completeness and for introducing notation. On the other hand steps 3 and 4 include some additional aspects in comparison to \cite{Ais11}.\\[12pt]
\textit{Step 1:} Define subsets $A_k$ and bound their cardinality by Proposition~\ref{prop:improvement_gnewuch}.\\
Let $$K:= \left\lceil \tfrac{\log_2(N)-\log_2(s)}{2} \right\rceil.$$ Then $K \geq 3$ and 
$$2^{-K} \in \left[ \frac{\sqrt{s}}{2\sqrt{N}}, \frac{\sqrt{s}}{\sqrt{N}} \right].$$
By Proposition~\ref{prop:improvement_gnewuch}, there exists a $2^{-k}$-cover of $[0,1]^s$ for $1 \leq k \leq K-1$, denoted by $\Gamma_k$, such that Stirling's formula yields
$$|\Gamma_k| \leq 2^{s-1} \left( \frac{s^s}{s!} \right) (2^k+1)^s + (2^k+1) \leq \frac{1}{\sqrt{2\pi s}} 2^{s-1} \exp(s)(2^k+1)^s + (2^k+1),$$
because $\mathcal{N}(s,d) \leq 2 N_{[\ ]}(s,d)$. Analogously, there exists a $2^{-K}$-bracketing cover $\Delta_K$ with 
$$|\Delta_K| \leq  \frac{1}{\sqrt{2\pi s}} 2^{s-2}\exp(s)(2^K+1)^s + (2^K+1).$$
Moreover we set
$$\Gamma_K := \left\{ v \in [0,1]^s \ | \ (v,w) \in \Delta_K \ \textrm{for some} \ w \right\}.$$
Fix $x \in [0,1]^s$ arbitrarily. We want to canonically define two sequences $v_k,w_k$ with $v_k,w_k \in \Gamma_k \cup \left\{ 0 \right\}$ such that
$$0 \leq v_1 \leq v_2 \leq \ldots \leq v_{K-1} \leq v_K \leq x \leq w_K$$% \leq w_2 \leq \ldots \leq w_{K-1} \leq w_K$$
holds. First we choose $v_K,w_K = (v_K(x),w_K(x))$ with $v_K \leq x \leq w_K$ and $\lambda_s[v_K,w_K] \leq 2^{-K}.$ For every $k, \ 2 \leq k \leq K$ and $\gamma \in \Gamma_{k}$, there exist $v_{k-1} = v_{k-1}(\gamma)$ and $w_{k-1} = w_{k-1}(\gamma)$ with $v_{k-1},w_{k-1} \in \Gamma_{k-1} \cup \left\{ 0 \right\}, v_k \leq \gamma \leq w_k$ and $\lambda_s[v_k,w_k] \leq 2^{-k+1}.$ Recursively we set $p_K(x) =: v_K(x)$ and $p_k: = v_{k}(v_{k+1}(x))$ for $1 \leq k \leq K-1$. Moreover we define $p_0 = 0$. Finally, for $x,y \in [0,1]$ let
$$\overline{[x,y]} := \begin{cases} [0,y] \setminus [0,x] & \textrm{if } x \neq 0\\ [0,y] & \textrm{if } x=0, y \neq 0\\ 0 & \textrm{if } x=y=0 \\\end{cases}.$$
For $0 \leq k \leq K-1$, the sets $[p_k(x),p_{k+1}(x)]$ are bounded by
$$\lambda_s \overline{[p_k(x),p_{k+1}(x)]} \leq 2^{-k}$$
and $[p_K(x),w_K(x)]$ by
$$\lambda_s \overline{[p_K(x),w_K(x)]}\leq 2^{-K}.$$
We define $A_k$ as the set of all sets of the form $[p_k(x),p_{k+1}(x)]$ for $0\leq k \leq K-1$ and $A_K$ as the set of all sets of the form $[p_K(x),w_K(x)]$. It was proven in \cite{Ais11} that $\lambda_s(A_k) \leq 2^{-k}$ for all $0 \leq k \leq K$. Moreover, every $p_{k+1} \in \Gamma_{k+1}$ is contained in some $A_k$ and hence $|A_k| \leq |\Gamma_{k+1}|$. Therefore we have 
$$|A_k| \leq \frac{1}{\sqrt{2\pi s}} 2^{s-1} \exp(s)(2^{k+1}+1)^s + (2^{k+1}+1)$$
and
$$|A_K| \leq \frac{1}{\sqrt{2\pi s}} 2^{s-2} \exp(s)(2^{K}+1)^s + \frac{1}{2}(2^{K}+1).$$
\textit{Step 2:} Calculate lower bound for expected value of indicator functions.\\
Let $X_1,\ldots,X_n$ be a sequence of i.i.d random variables defined on some probability space $(\Omega,\mathcal{A},\mathbb{P})$ having uniform distribution on $[0,1]^s$. For $I \in A_k$ set $Z_i:= \mathbb{1}_I(X_i) - \lambda_s(I).$ In \cite{Ais11}, it is shown by using Bernstein's and Hoeffding's inequality that for arbitrary $c> 0$ 
\begin{align*}
\mathbb{P} \left( \left| \sum_{i=1}^N Z_i \right|  > c \sqrt{sN} \right) \leq \begin{cases} 2\exp \left( - \frac{c^2s}{2^{-k+1}(1-2^{-k})+4c^22^{-K}/3}\right) & \textrm{for} \ 2 \leq k \leq K \\ 2\exp(-2c^2s) & \textrm{for} \ k=0,1 \end{cases}
\end{align*}
\textit{Step 3:} Show that 
$$\mathbb{P}\left( \bigcup\limits_{k=0}^K \ \bigcup\limits_{I\in A_k} \left( \left| \sum\limits_{i=1}^N \mathbb{1}_I(X_i)-N\lambda_s(I) \right|>c_k\sqrt{sN} \right) \right)<1,$$
where the coefficients $c_k$ will be chosen in the following.\\
Let 
$$B_k:=\bigcup\limits_{I\in A_k} \left( \left| \sum\limits_{i=1}^N \mathds{1}_I(X_i)-N\lambda_s(I) \right|>c_k\sqrt{sN} \right).$$
For $k=0$ we get
\begin{align*}
\mathbb{P}(B_0) & \leq 2\exp(-2c_0^2s) |A_0|\\
& \leq 2\exp(-2c_0^2s) \left( \frac{1}{\sqrt{2\pi s}}2^{s-1}\exp(s)3^s + 3 \right)\\
& \stackrel{s \geq 3}{\leq} \exp(-2c_0^2s) \exp(s) 6^s \left( \frac{1}{\sqrt{6 \pi}} + 6 \exp(-3)6^{-3} \right)\\
& \leq \frac{1}{4} \exp(-2c_os^2)\exp(s)6^s.
\end{align*}
Thus $\mathbb{P}(B_0) \leq 1/4$ holds after choosing $$c_0 = \sqrt{\tfrac{\log(6)+1}{2}} \leq 1.19.$$
Analogously, we choose $$c_1 = \sqrt{\tfrac{\log(10)+1}{2}} \leq 1.29$$ and get $\mathbb{P}(B_1) \leq 1/4$. Now let $2 \leq k \leq K-1$. Then
\begin{align*}
2^k\mathbb{P}(B_k) \leq & \underbrace{\left( \frac{2^{s-1}}{\sqrt{2\pi s}} \exp(s) (2^{k+1}+1)^s + (2^{k+1}+1) \right) 2^{k+1}}_{=:G} \cdot \\
 & \quad \exp\left(-\frac{c_k^2s}{2^{-k+1}(1-2^{-k}) + \tfrac{4c_k}{3}2^{-K}} \right)
\end{align*}
At first we bound $G$ by
\begin{align*}
	G & \stackrel{s \geq 3}{\leq} \left( \frac{1}{2\sqrt{6 \pi}} + 2^{-3}(2^{k+1}+1)^{-2}\exp(-3) \right) 2^{s+1+k} (2^{k+1}+1)^s\exp(s)\\
	& \leq \frac{1}{8}  2^{s+1+k} (2^{k+1}+1)^s\exp(s) \\
	& \leq \frac{1}{8} 2^{k+1} \exp(s(\log(2)+1+\log(2^{k+1}+1)))\\
	& \stackrel{s \geq 3}{\leq} \exp\left(s\underbrace{\left(\frac{4}{3}(1+k)\log(2)+1+\log(2^{-k-1}+1)\right)}_{:=a_k}\right).
\end{align*}
Finally we define $c_k$ as the positive solution of the equation 
$$c_k = \sqrt{a_k} \cdot \sqrt{2^{-k+1}(1-2^{k})+\tfrac{4}{3}c_k2^{-K}},$$ which yields $|c_k| \leq 1.58$ and
$$2^k\mathbb{P}(B_k) \leq \exp(s\cdot a_k) \exp\left(-\frac{c_k^2s}{2^{-k+1}(1-2^{-k}) + \tfrac{4c_k}{3}2^{-K}} \right) \leq 1$$
and thus $\mathbb{P}(B_k) \leq 2^{-k}$. For $k=K$ the set $A_K$ contains at most
$$|A_K| \leq |\Delta_K| \leq \frac{1}{\sqrt{2\pi s}} 2^{s-2}\exp(s)(2^K+1)^s + (2^K+1)$$
elements. Similarly to the last case we obtain
\begin{align*}
	2^K\mathbb{P}(B_K) \leq & \exp\left( s \left( \underbrace{\frac{4}{3}K\log(2)+1+\log(1+2^{-K})}_{:=a_K} \right) \right) \cdot\\
	& \quad \exp\left(-\frac{c_K^2s}{2^{-K+1}(1-2^{-K}) + \tfrac{4c_K}{3}2^{-K}} \right).
	%HP
\end{align*}
Defining $c_K$ via the equation
$$c_K = \sqrt{a_K} \cdot \sqrt{2^{-K+1}(1-2^{-K})+\frac{4}{3}c_K2^{-K}}$$
we arrive at $|c_K| \leq 1.33$ and $\mathbb{P}(B_K) \leq 2^{-K}.$ This completes step 3.\\[12pt]
\textit{Step 4:} Show that 
\begin{align} \label{inequ:ck}
\sum_{k=0}^K c_k < 8.
\end{align}
The fact that the choice of the $c_k$ depends on $K$, will be reflected by the notation $c_{k,K}$ in this step. For $K \leq 31$, the desired inequality can be checked by computer calculation. In this range, the maximal value is achieved for $K = 31$ and $\sum_{k=0}^{31} c_{k,31} \leq 7.99789995$. Hence let $K \geq 32$. Since the $c_{k,K}$ are monotonically decreasing for increasing $K$, we have
$$\sum_{k=0}^{31} c_{k,K} \leq 7.99789995$$
and $c_{K,K} \leq c_{32,32} \leq 5 \cdot 10^{-9}$ for $K\geq 32$. Solving the equation that defines $c_{k,K}$, we find 
$$c_{k,K} \leq 0.2480726 \cdot k 2^{-k/2}$$
for $k,K \geq 32$, $k\leq K-1$ and thus end up with the desired bound.\\[12pt]
\textit{Step 5:} Derive inequality~\eqref{inequ:main} \\
According to step 3 we may choose a realization $X_1(\omega),\ldots,X_N(\omega)$ with
$$\omega \notin \bigcup_{k=0}^K B_k$$
and set $z_n := X_n(\omega)$ for $1 \leq n \leq N$. In \cite{Ais11}, it is proven that
$$N\lambda_s([0,x]) - \left( \sum_{k=0}^{K-1} c_k + 1 \right) \sqrt{sN} \leq \sum_{n=1}^N \mathbb{1}_{[0,x)}(z_n) \leq N\lambda_s([0,x]) + \left( \sum_{k=0}^{K-1} c_k + 1 \right)\sqrt{sN}$$
holds for arbitrary $x \in [0,1]^s$. Thus~\eqref{inequ:main} follows from \eqref{inequ:ck}.
 \end{proof} 

%\tableofcontents

\paragraph{Acknowledgments.} The first-named author thanks Markus Weimar for supervising his master thesis which this paper builds on and R\"udiger Verf\"urth for his constant support. The second-named author did parts of the work on this paper during a stay at the Fields Institute whom he would like to thank for hospitality.

Hendrik Pasing\\
\textsc{Hochschule Ruhr West, Duisburger Str. 100, D-45479 M\"ulheim an der Ruhr}\\
\textit{E-mail address:} \texttt{hendrik.pasing@hs-ruhrwest.de}\\[12pt]
Christian Wei\ss\\
\textsc{Hochschule Ruhr West, Duisburger Str. 100, D-45479 M\"ulheim an der Ruhr}\\
\textit{E-mail address:} \texttt{christian.weiss@hs-ruhrwest.de}

\end{document}